\documentclass[11pt]{amsart}
\usepackage[foot]{amsaddr}
\usepackage[utf8x]{inputenc}
\usepackage{color}
\usepackage{amsmath,amsfonts}
\usepackage{amsfonts,amsmath,amsthm,amssymb,latexsym}
\usepackage{tikz}
\usepackage{mathrsfs}

\usepackage{algorithm}
\usepackage{algorithmic}
\usepackage{ulem}

\usepackage{hyperref}

\newtheorem{theorem}{Theorem}[section]
\newtheorem{proposition}[theorem]{Proposition}
\newtheorem{lemma}[theorem]{Lemma}

\theoremstyle{definition}

\newtheorem{example}[theorem]{Example}

\theoremstyle{remark}
\newtheorem{remark}[theorem]{Remark}
\numberwithin{equation}{section}
\theoremstyle{remark}

\DeclareMathOperator{\num}{num}
\DeclareMathOperator{\denom}{denom}

\newcommand{\C}{\mathbb{C}}

\newcommand{\Q}{\mathbb{Q}}

\newcommand{\N}{\mathbb{N}}
\newcommand{\FF}{\mathbb{F}}
\newcommand{\To}{\mathbb{T}}
\newcommand{\V}{\mathbb{V}}
\newcommand{\Pa}{\mathcal{P}}
\newcommand{\Qa}{\mathcal{Q}}
\newcommand{\sys}{\mathcal{F}}
\newcommand{\G}{\mathcal{G}}
\newcommand{\B}{\mathcal{B}}
\newcommand{\im}{\mathrm{Im}}

\newcommand\para{\vspace*{2mm}}

\allowdisplaybreaks

\begin{document}

{\color{red} The journal version of this paper appears in

\para

\noindent
\textit{ Sebastian Falkensteiner, J. Rafael Sendra. Transforming Radical Differential Equations to Algebraic Differential Equations.\\
Journal reference: Mediterr. J. Math. (2024) 21:87.\\  
Related DOI: \url{ https://doi.org/10.1007/s00009-024-02624-1}}

\para

\noindent
licensed under a Creative Commons Attribution 4.0 International License, \url{https://creativecommons.org/licenses/by/4.0/}}

\vspace*{1cm}

\date{\today}
\title{Transforming Radical Differential Equations to Algebraic Differential Equations}

\author{Sebastian Falkensteiner}
\address{Research Institute for Symbolic Computation (RISC), Johannes Kepler University Linz, Austria.}
\email{falkensteiner@risc.jku.at}

\author{J.Rafael Sendra}
\address{Department of Quantitative Methods, CUNEF-Universidad, Spain}
\email{jrafael.sendra@cunef.edu}

\begin{abstract}
In this paper we present an algorithmic procedure that transforms, if possible, a given system of ordinary or partial differential equations with radical dependencies in the unknown function and its derivatives into a system with polynomial relations among them by means of a rational change of variables. 
The solutions of the given equation and its transformation correspond one-to-one. 
This work can be seen as a generalization of previous work on reparametrization of ODEs and PDEs with radical coefficients.
\end{abstract}

\maketitle


\vspace*{2mm}


\section*{Acknowledgements}
Authors partially supported by the grant PID2020-113192GB-I00 (Mathematical Visualization: Foundations, Algorithms and Applications) from the Spanish MICINN. 
Second author also partially supported by the OeAD project FR 09/2022.

\section{Introduction}
Some differential equations appear to be easier solvable than others. 
A common method for finding solutions is to first transform the given differential equation into another one where solution methods are known, find its solutions, and transform them back accordingly. 
In applied mathematics, often the given equation is transformed into an explicit differential equation where the highest-order derivatives are explicitly written as a function of the independent variables, dependent variables, and lower order derivatives. 
Then classic results such as the Cauchy-Kovalevskaya theorem are applied in order to show existence and uniqueness of solutions with given initial data fulfilling certain criteria. 
Various other approaches for computing exact solutions exist for algebraic differential equations (shortly ADEs). 
They are polynomials in the unknowns and their derivatives, and the coefficients of the polynomial are rational in the independent variables, see e.g.~\cite{ritt1950differential}. 
In this paper, we show a strategy to transform differential equations involving radical expressions to such ADEs. 
For the sake of simplicity we mainly focus on ordinary differential equations, but also give insight into the partial case.

In~\cite{caravantes2021transforming}, the authors focus on AODEs whose coefficients involve radical expressions of the independent variable $x$, and they study the possible transformation of the equation into another with only rational expressions of $x$. 
Moreover, they show how to relate the solutions of both equations. 
For instance, the method in~\cite{caravantes2021transforming} transforms
\[\sqrt{x}\,y'^3y''+y+x=0\]
into
\[z\,Y'^3Y''-Y'^4+8z^3\,Y+8z^5=0\]
by using the change $r(z)=z^2, y(z^2)=Y(z)$. 
In this paper, we take a step forward and we analyze differential equations where the unknown function and its derivatives can appear as radical expressions as for instance 
\[\sqrt{y'^2+y^2}+x\,\sqrt[3]{y}^2-yy'=2x.\]
More precisely, we present a method to find, if they exist, rational changes of variables that convert systems of ODEs and PDEs of this type into systems of algebraic differential equations. 
Similarly as in~\cite{caravantes2021transforming}, the underlying idea is to associate to the differential system a radical parametrization such that the possible change of variable can be derived from the  tower variety; we refer to~\cite{sendra2011radical,sendra2013first,sendra2017algebraic} for further details on radical varieties.

Combining the results in this paper with that in~\cite{caravantes2021transforming}, we can analyze mixed types of equations where radical expressions occur in the differential indeterminates and, separately, in the differential determinates. 
An example of a single equation of this type is the Thomas-Fermi equation, cf. also Example \ref{example-fermi} below.  
In this type of situations, one can successively apply the two transformations, the one that affects the independent variables and the one that involves the dependent ones, to solve the radicalness of the equations. 
One may apply the  transformations in any of the two orders. Nevertheless, 
the transformation on the independent variables keeps the number of equations, while the transformation on the dependent variable 
will introduce, in general, additional equations. Thus, the suggestion would be to  proceed first to resolve the radicals involving the independent variables.  

The main contribution of the paper is the development of a method that transforms the given equation into a system of algebraic differential equations. Clearly the benefits of considering algebraic differential
equations instead of differential equations, with radical dependencies, depend on the particular type of algebraic differential equations that the method generates in each particular application.  Nevertheless,  the expected benefit is that   the existing techniques for computing exact solutions of algebraic differential equations will be available (see e.g. \cite{amsBulletin}). In addition,
there have been several approaches in order to transform the equation and approximate its solutions, see for instance~\cite{burrows1984variational,turkyilmazoglu2012solution}. 
By our method, symbolic solutions of the transformed equations can be computed by well-known techniques for AODEs such as the Newton polygon method~\cite{Cano2005}.  In Example \ref{ex-maple}, we show a differential equation where the application of the method in this paper generates  more satisfactory  expressions of the solutions than the direct treatment on the given equation; for this, the example has been developed with the aid of the computer algebra system  Maple 2023.

\vspace*{2mm}

The problem that we study is the following. 
Define the radical tower
\begin{equation}\label{eq-tower}
\FF_0 = \C(z_1,\ldots,z_n) \subseteq \FF_1 \subseteq \cdots \subseteq \FF_m,
\end{equation}
where $\FF_i = \FF_{i-1}(\delta_i) = \FF_0(\delta_1,\ldots,\delta_i)$ with $\delta_i^{e_i} = \alpha_i \in \FF_{i-1}, e_i \in \N$. 
In other words, $\FF_m$ is the field of functions built using several (possibly nested) radicals.
Consider the system of differential equations
\begin{equation}\label{eq-system}
\sys = \{ F_1(x,y_1,\ldots,y_\ell^{(n_\ell)})=0, \ldots, F_N(x,y_1,\ldots,y_\ell^{(n_\ell)})=0 \}
\end{equation}
with $n=n_1+\cdots+n_\ell+\ell$, $F_i(x,z_1,\ldots,z_n) \in \FF_m(x)$.
In other words, the independent variable $x$ appears rationally and the dependent variables $y_1,\ldots,y_\ell$ and their derivatives possibly appear inside radicals. 
We may assume that $F_i$ do not have any denominators. 
Otherwise we consider their numerators $\num(F_i)$ and we assume that the solutions are not zeros of the denominators.
Our goal is to find, if it exists, a rational change of variables of the form
\begin{equation}\label{eq-substitution}
z_1 = r_1(w_1,\ldots,w_n), \ldots, z_n = r_n(w_1,\ldots,w_n)
\end{equation}
such that the new equations
\begin{equation}\label{eq-new}
G_i(x,w_1,\ldots,w_n):=F_i(x,r_1(w_1,\ldots,w_n),\ldots,r_n(w_1,\ldots,w_n))=0
\end{equation}
have rational expressions in $w_1,\ldots,w_n$ and polynomials in $x$ as coefficients. 
The denominator of $G_i$ is a multiple of the denominators of $r_1,\ldots,r_n$. 
When speaking about a solution of $G_i=0$, we again consider
\begin{equation*}
\num(G_i)(x,w_1,\ldots,w_n)=0
\end{equation*}
and we assume that the denominator of $G_i$ does not vanish when substituting in the solution.

\vspace*{2mm}

The paper is structured as follows.
The next section describes the algebraic preliminaries on radical varieties that will be used later. 
Section~\ref{sec-ODE} describes the results on systems of ordinary differential equations. 
Section~\ref{sec-special} deals with special cases of the given system. 
In Section~\ref{sec-alg} the process for transforming the differential equation, with radicals, into a system of ADEs is summarized in a procedure. 
In the situations described in Section~\ref{sec-special}, this procedure is indeed algorithmic. 
We underline this by giving some examples. 
Section~\ref{sec-PDE} treats the case of systems of partial differential equations.

\section{Preliminaries on Radical Varieties} \label{sec-pre}
In this section we recall some results on radical varieties; for further details see~\cite{sendra2017algebraic}. 
For the basic notions and results on algebraic geometry, which are fundamental here, see for example~\cite{Cox1}.

We will express tuples of variables or functions by using bars such as $\bar z = (z_1,\ldots,z_n)$.
Let $\FF_m$ and $\delta_i$ be defined as in~\eqref{eq-tower}. 
A radical parametrization of~\eqref{eq-tower} is a tuple
$$\Pa = (P_1(\bar z),\ldots,P_k(\bar z))$$
of elements of $\FF_m$ whose Jacobian has rank $n$. 
For a suitable election of branches we obtain a (usually non-rational) function with domain in $\C^{n}$ and whose image is Zariski dense in an $n$-dimensional variety. 
The Zariski closure of the image $\im(\Pa)$ is the (irreducible) radical variety determined by $\Pa$ and will be denoted by $\V(\Pa)$~\cite[Theorem 3.11 (iii)]{sendra2017algebraic}.

The function $\Pa$ can be factored as $R \circ \psi$, where $\psi(\bar{w})=(\bar{w},\bar{\delta}(\bar{w}))$, $\bar \delta$ are the radicals in the construction of the tower~\eqref{eq-tower}, and $R$ is a rational function.
Another relevant algebraic variety in our construction is the tower variety of $\Pa$, defined as the Zariski closure of $\im(\psi)$, denoted by $\V(\To)$. 
It is also irreducible of dimension $n$. 
Therefore, the map $R: \V(\To) \dashrightarrow \V(\Pa)$ is rational.
The tower variety contains useful information about $\Pa$ and its properties are essential in this work.
One more ingredient is an incidence variety, the Zariski closure of the map $\bar z \mapsto (\bar z,\bar \delta, \bar q)$, where $\bar q = (q_1,\ldots,q_k)$ represents the coordinates of the points in $\V(\Pa)$, that we denote as $\B(\Pa)$. 
This provides the projections $\pi: \B(\Pa) \dashrightarrow \V(\Pa)$ and $\pi^*: \B(\Pa) \dashrightarrow \V(\To)$. 
In~\cite{sendra2017algebraic}, generators of $\B(\Pa)$ are given and hence, using the Closure theorem (see~\cite{Cox1}), one may compute generators of the varieties $\V(\Pa)$ and $\V(\To)$.
We define the tracing index of $\Pa$ as the degree of the map $\pi$, that is the cardinality of the generic fiber of $\pi$. 
The calculation of the tracing index is described in~\cite{sendra2017algebraic}. 
The case where the tracing index is equal to one is of particular interest, because in this case $\V(\Pa)$ and $\V(\To)$ are birationally equivalent~\cite[Theorem 4.11]{sendra2017algebraic}.


\section{Differential Equations}\label{sec-ODE}
In this section, we will use some of the notation introduced previously. 
In particular, $\FF_m$ is the last field of a radical tower over $\C(z_1,\ldots,z_n)$, introduced in~\eqref{eq-tower}, and $\bar \delta= (\delta_1,\ldots,\delta_m)$ is the tuple of algebraic elements used in the construction of the tower. 
We also consider the system of differential equations introduced in~\eqref{eq-system}.
After the change of variables
\[\bar z= \bar r(\bar w)= (\bar r_1(\bar w),\ldots,\bar r_\ell(\bar w)), \]
where 
\[ \bar r_i=(r_{i,0}(\bar w),\ldots,r_{i,n_i}(\bar w)), \]
we consider $G_i(x, \bar w) = F_i(x,\bar r(\bar w))$ as in~\eqref{eq-new}. 
By a change of variable we allow mappings $\bar r(\bar w)$  whose Jacobian is generically non-singular, so that, via the Inverse Function Theorem, the transformation can be inverted. 
Recall that $G_i=0$ is an algebraic equation depending on $x, \bar w$ and no derivatives, and involves fractional expressions in $\bar w$. 
The differential relations among the variables are kept by additionally imposing, for $1 \le i \le \ell, 0 \le j < n_i$, the equations
\begin{equation*}\label{eq-Jacobian}
r_{i,j+1}(\bar w) = J_{r_{i,j}}(\bar w) \cdot (\bar w')^T = J_{r_{i,j}}(w_1,\ldots,w_n) \cdot (w_1',\ldots,w_n')^T
\end{equation*}
where $J_{r_{i,j}}$ denotes the Jacobian of $r_{i,j}$. 
Then we obtain the differential system
\begin{equation}\label{eq-system-new}
\G = \begin{cases}
G_1(x,\bar w) = 0, \ldots, G_N(x,\bar w) = 0 \\
r_{i,j+1}(\bar w) = J_{r_{i,j}}(\bar w) \cdot (\bar w')^T, \ \ 1 \le i \le \ell, 0 \le j < n_i
\end{cases}
\end{equation}

\begin{remark}\label{rem:denom}
	The algebraic equations $G_i(x,\bar w)=0$, and the differential equations $r_{j+1}(\bar w) = J_{r_j}(\bar w) \cdot (\bar w')^T$ in~\eqref{eq-system-new} might have denominators. 
	In this case, all denominators depend only on $\bar w$ and are independent of $x$. 
	Moreover, the denominator of $G_i$ arises from the change of variables in the $F_i$ and hence, it is a product of powers   of the denominators of the $r_{i,j}$. 
	Consequently, for a given function $\bar w(x)$, if the substitution $r_{j}(\bar w(x))$ is well--defined for every $0 \le j \le n$, then the denominator of $G_i(x,\bar w(x))$ is non-zero as well.
\end{remark}

\begin{remark}\label{rem:elimination}
	After possibly clearing denominators in~\eqref{eq-system-new}, the system $\G$ might be simplified by various differential elimination methods such as the Thomas decomposition or differential regular chains. 
	For a description of these two methods see~\cite{robertz2014formal,boulier2019equivalence}.
\end{remark}

\begin{remark}\label{rem:naive}
	For given $\sys$ as in~\eqref{eq-system}, an alternative approach to find a system of ADEs would be to directly replace every radical expression by a new variable. 
	More precisely, let $\bar{H}(\bar{z})$ be the tuple of non-rational coefficients of the $F(x,\bar{z}) \in \sys$ w.r.t. $x$ and set $\bar{w}:=\bar{H}(\bar{z})$. 
	Let $G_i(x,\bar{z},\bar{w})$ be the resulting ADEs after replacing $\bar{H}(\bar{z})$ by $\bar{w}$. 
	Additionally, by taking suitable powers, the equations $\bar{w}=\bar{H}(\bar{z})$ can be transformed into algebraic equations of the form $\bar{T}(\bar{z},\bar{w})=0$. 
	Then  $$\mathcal{G}=\{ G_1(x,\bar{z},\bar{w}),\ldots,G_N(x,\bar{z},\bar{w}), \bar{T}(\bar{z},\bar{w}) \}$$
	is a system of ADEs. 
	An obstacle with this approach is that the number of variables increases from $n$ to $n+\#(\bar{H})$. 
	In the end, we want to solve $\mathcal{G}$ only for $\bar{z}$. 
	For this purpose, it would be reasonable to apply a differential elimination method to $\mathcal{G}$ (with respect to an ordering where the $z_i$ are ranked lower than the $w_i$). 
	In most examples, however, this elimination step does not terminate in reasonable time.
	
	In contrast to this more direct approach, we use parametric expressions of the $\bar{w}=\bar{H}(\bar{z})$ in order to keep the number of variables. 
	In this way, the resulting system~\eqref{eq-system-new} does not involve more variables than the original system~$\sys$.
\end{remark}

The next theorem states sufficient conditions of $\bar r(\bar w)$ such that~\eqref{eq-system-new} is a system of AODEs.

\begin{theorem}\label{thm:RationalTransformation}
	Let $\sys \subset \FF_m(x)$ be as in~\eqref{eq-system}, let $\bar r(\bar w)\in \C(\bar w)^{n}$ be a rational change of variables, and let $\bar H(\bar z) \in \FF_m^\rho$ be the tuple of non-rational (function) coefficients of the $F(x,\bar z) \in \sys$ w.r.t. $x$, taken in any order. 
	The following are equivalent.
	\begin{enumerate}
		\item \eqref{eq-system-new} is a system of AODEs.
		\item All the components of $\bar H(\bar r(\bar w))$ are rational.
	\end{enumerate}
\end{theorem}
\begin{proof}
	``(1)$\implies$(2)'': Since~\eqref{eq-system-new} is a system of AODEs, in particular, the $G_i(x,\bar w)$ are algebraic. 
	Thus $F_i(x,\bar r(\bar w))=0$ is an algebraic equation for every $F_i \in \sys$. 
	So its coefficients, and in particular $\bar H(\bar r(\bar w))$, are rational functions.
	
	\noindent ``(2)$\implies$(1)'': After the substitution $\bar z = \bar r(\bar w)$, every $F_i(x,\bar r(\bar w))$ has coefficients which are rational in $\bar w$ and polynomial in $x$, so $G_i(x,\bar w)$ is algebraic. 
	Moreover, since $\bar{r}$ is a tuple of rational functions, the remaining equations of~\eqref{eq-system-new} are AODEs.
\end{proof}

In the sequel we associate to the differential system~\eqref{eq-system} a radical pa\-ra\-me\-tri\-za\-tion, namely
\begin{equation}\label{eq-par}
\Pa = (\bar z, \bar H(\bar z))
\end{equation}
with $\bar H$ as in Theorem~\ref{thm:RationalTransformation}, and we consider its associated radical variety $\V(\Pa)$ as well as its tower variety $\V(\To)$, see Section~\ref{sec-pre}; observe that, as required, the rank of the Jacobian of $\Pa$ is $n$.

\begin{lemma}\label{lem-unirational}
	Let $\V(\To)$ and $\V(\Pa)$ be as above. 
	Then the following statements hold.
	\begin{enumerate}
		\item $\dim(\V(\To))=\dim(\V(\Pa))=n$.
		\item If $\V(\To)$ is unirational, then $\V(\Pa)$ is unirational.
		Moreover, if the tracing index of $\Pa$ is one and $\V(\To)$ is rational, then $\V(\Pa)$ is rational.
	\end{enumerate}
\end{lemma}
\begin{proof}
	By Theorem 3.11(iii) and Theorem 4.6 in~\cite{sendra2017algebraic}, the dimension of both, $\V(\To)$ and $\V(\Pa)$, is equal to $n$. 
	The second item follows from Theorem 4.11 in the same reference.
\end{proof}

In the case where $\V(\To)$ is a unirational curve or surface, by Lemma~\ref{lem-unirational}, $\V(\Pa)$ is unirational and, by L\"uroth's Theorem for curves and surfaces, $\V(\To)$ and $\V(\Pa)$ are rational. 
Indeed, there exist algorithmic approaches to compute such birational parametrizations~\cite{SWP08,S1998Rational}.
In Lemma~\ref{lem-unirational} we have shown that the dimension of the tower variety $\V(\To)$ and $\V(\Pa)$ is equal to $n=n_1+\cdots+n_\ell+\ell$, where $\ell$ is the number of differential indeterminates and $n_i$ is the order with respect to $y_i$, but as we will show in Section~\ref{sec-alg}, in various situations the dimension can be reduced and the case where these varieties are curves or surfaces might occur.

\begin{theorem}\label{thm-solutions}
	Let $\sys$ be as in~\eqref{eq-system}, $\G$ be as in~\eqref{eq-system-new} and let $\Qa(\bar w)=(\bar r(\bar w),\bar \delta(\bar r (\bar w)))$ be a rational parametrization of $\V(\To)$. 
	It holds that
	\begin{enumerate}
		\item For every solution $\bar w(x)$ of $\G$, $$(y_1(x),\ldots,y_\ell(x)):=(r_{1,0}(\bar w(x)),\ldots,r_{\ell,0}(\bar w(x)))$$ is a solution of the differential system $\sys$. 
		\item If $\Qa$ is invertible, $(y_1(x),\ldots,y_\ell(x))$ is a solution of $\sys$, $$\bar w(x):= \Qa^{-1}(y_1(x),\ldots,y_\ell^{(n_\ell)}(x),\bar \delta(y_1(x),\ldots,y_\ell^{(n_\ell)}(x)))$$ and $\Qa (\bar w(x))$ are well--defined, then $\bar w(x)$ is a solution of $\G$.
	\end{enumerate}
\end{theorem}
\begin{proof}
	Using the differential equations in $\G$, one gets for $1 \le i \le \ell, 1 \le j < n_{i}$ that $r_{i,j}(\bar w(x))$ is the $j$-th derivative of $r_{i,0}(\bar w(x))$. 
	So, $\bar r(\bar w(x))=(y_1(x),\ldots,y_\ell^{(n_\ell)}(x))$. 
	Moreover, for every $1 \le k \le N$, since $\num(G_k)(x, \bar w(x))=0$, and, by Remark~\ref{rem:denom}, the denominator of $G_k$ specialized at $(x, \bar w(x))$ is non-zero, so
	$$F_k(x,\bar r(\bar w(x)))=G_k(x,\bar w(x))=0.$$
	
	Conversely, let $(y_1(x),\ldots,y_\ell(x))$ be a solution of $\sys$, and let $T_1,\ldots, T_t$ be generators of the tower variety; that is, $\V(\To)$ is the zero-set of $\{T_1,\ldots,T_t\}$.
	Let $\bar u=(u_1,\ldots,u_n)$ and $\bar v=(v_1,\ldots,v_m)$ be such that $(\bar u,\bar v)\in \V(\To)$ and such that $\Qa(\Qa^{-1}(\bar u, \bar v))$ is well--defined. 
	Then,
	\[ \Qa(\Qa^{-1}(\bar u, \bar v))=(\bar u, \bar v) \,\,\mod \,\, I(\V(\To)). \]
	In particular,
	\[ \bar r(\Qa^{-1}(\bar u, \bar v))=\bar u \,\,\mod \,\, I(\V(\To)). \]
	Therefore, there exist polynomials $M_{i,1},\ldots,M_{i,t}$, for $i\in \{0,\ldots,n \}$, such that
	\begin{align*} \bar r(\bar w(x))&=(y_1(x),\ldots, y_\ell^{(n_\ell)}(x))\\
	&+\left(\sum_{j=1}^{t} M_{0,j}(\bar \sigma(x)) \,T_{j}(\bar \sigma(x)), \ldots,\sum_{j=1}^{t} M_{n,j}(\bar \sigma(x)) \, T_{j}(\bar \sigma(x)) \right).
	\end{align*}
	where $\bar \sigma(x):= (y_1(x),\ldots,y_\ell^{(n_\ell)}(x),\bar \delta(y_1(x),\ldots,y_\ell^{(n_\ell)}(x)))$. 
	Furthermore, $\bar \sigma(x)=\psi (y_1(x),\ldots,y_\ell^{(n_\ell)}(x))\in \im(\psi) \subset \V(\To)$ for all $x$.
	Thus, $T_i(\bar \sigma(x))=0$ for all $i\in \{1,\ldots,t\}$, and hence
	\[ \bar r(\bar w(x))=(y_1(x),\ldots, y_\ell^{(n_\ell)}(x)).\]
	In this situation, we first observe that since $\Qa(\bar w(x))$ is well--defined, then the denominators of $\Qa$ do not vanish at $\bar w(x)$ and, by Remark~\ref{rem:denom}, $G_k(x,\bar w(x))$ is also well--defined for every $1 \le k \le N$. 
	In addition,
	\[ G_k(x,\bar w(x))=F_k(x, \bar r (\bar w(x)))=F_k(x,y_1(x),\ldots,y_\ell^{(n_\ell)}(x))=0. \]
	Now, clearly $\num(G_k)(x,\bar w(x))=0$. 
	The remaining equations in $\G$ hold by the chain rule applied to the components of $\bar r_{i}(\bar w(x))$, i.e. for $1 \le i \le \ell$, $0 \le j <n_i$
	$$r_{i,j+1}(\bar w(x))=\frac{d}{dx}(y_i^{(j)}(x))=\frac{d}{dx}(r_{i,j}(\bar w(x)))=J_{r_{i,j}}(\bar w(x)) \cdot (\bar w'(x))^T.$$
\end{proof}

Note that in Theorem~\ref{thm-solutions} we assume that $\bar r(\bar w(x))$ is well--defined. 
In the case that the tower variety $\V(\To)$ admits a polynomial parametrization $\Qa$, this is always the case. 
For details on polynomial parametrizations see e.g.~\cite{SendraNormalidad,perez2020computing}.

 In the following proposition we analyze the particular case where the system contains a single equation of the form $F(y_1,y_1')=0$ or $F(y_1,y_2)=0$.

\begin{proposition}\label{prop:denom}
	Let $\sys$ be as in~\eqref{eq-system}, $\G$ be as in~\eqref{eq-system-new} and let $\Qa(\bar w)=(\bar r(\bar w),\bar \delta(\bar r (\bar w)))$ be a rational parametrization of $\V(\To)$. 
	If $\sys = \{F\}$ consists only of an autonomous equation, $n=2$, and $\bar w(x)=(w_1(x),w_2(x))$ is a zero of the denominator of $r_1(w_1,w_2)$ or $r_2(w_1,w_2)$ and of the numerator of $G$, then $\bar w(x)$ is constant.
\end{proposition}
\begin{proof}
	Let $d_j \in \C[w_1,w_2]$ be the denominator of $r_j$ for $j \in \{1,2\}$. 
	Since $F$ is independent of $x$, also $G$ and its numerator are autonomous. 
	Hence, $\num(G)(\bar w(x))=d_j(\bar w(x))=0$ if and only if either $\bar w(x)$ is constant or $\num(G), d_j \in \C[w_1,w_2]$ share a common component. 
	In the latter case, since $\denom(G)$ is a multiple of the $d_j$, also $\num(G)$ and $\denom(G)$ have a common factor, which can be neglected.
\end{proof}

\section{Special Cases}\label{sec-special}
In the substitution~\eqref{eq-substitution}, the number of new variables is equal to the sum of the orders of the given differential equations in the system~\eqref{eq-system} with respect to the differential indeterminates. 
In many cases the radicals in $\sys$ can be considered separately or even do not contain some of the derivatives of the $y_i$. 
In these cases, one may take a projection of the tower variety where the dimension is smaller than the order. 
In the following, we deal with these special cases.

\vspace*{1mm}

\noindent \textsf{Special case I.} We start with the case where there exist some variables $z_i$ which do not occur in any of the radicals in~\eqref{eq-system}.
Let $\bar H$ be as in Theorem~\ref{thm:RationalTransformation}.  
Let us say that $I \subsetneq \{1,\ldots,n\}$ is the index set of variables $z_i$ appearing in $\bar H$ and $I^*=\{1,\ldots,n\}\setminus I$. Say that $s=\#(I)$.
We will use the following notation: 
\begin{itemize}
	\item For an $n$--tuple $\bar p$ we denote by $\bar p^c$ the $s$--tuple obtained deleting the elements in the tuple positions labeled by $I^*$.
	\item For an $s$--tuple $\bar p$ and for an $(n-s)$--tuple $\bar{q}$ we denote by $\bar p^e(\bar q)$ the $n$--tuple $\bar s$ obtained inserting into  $\bar p$ the elements of $\bar q$ so that $\bar{s}^c=\bar p$.
\end{itemize}
We consider the projection 
\[ \begin{array}{lccc} 
\pi_{I}: & \V(\To)\subset \mathbb{C}^{n+m} & \longrightarrow & \mathbb{C}^{s+m} \\
& (\bar u, \bar v) &\longmapsto & (\bar u^c, \bar v)
\end{array}\]
where $\bar u=(u_1,\ldots,u_n), \bar v=(v_1,\ldots,v_m)$ and $\V(\To)$ is as in \eqref{eq-tower}.
Let $\V(\To)_I$ be the Zariski closure of $\pi_I(\V(\To))$. 
Observe that $\V(\To)_I$ is the tower variety of the radical parametrization $(\bar{z}^{c},\bar H(\bar{z}^{c}))$ and hence, by Lemma~\ref{lem-unirational}, $\dim(\V(\To)_I)=s<n$. 
Furthermore, note that for $\bar p:=(\bar a, \bar b) \in \pi_I(\V(\To))$ where $\bar a\in \C^{s}, \bar b\in \C^m$, it holds that
\[ \pi_{I}^{-1}(\bar p)=\{ (\bar{a}^e(\bar p),\bar b)\,|\, \bar p\in \mathbb{C}^{n-s}\}. \]
For an $n$--tuple $\bar p=(p_1,\ldots,p_n)$, and for $K\subset \{1,\ldots,n\}$ we denote by $\bar p^K$ the tuple $(p_i)_{i\in K}$.
Then, if $\bar r(\bar w):=(\bar r^I(\bar w^I))^e(\bar{w}^{I^*}),$ and if 
\begin{equation}\label{eq-A}
\tilde{\Qa}(\bar w^I):=(\bar r^I(\bar w^I), \bar \delta (\bar r^I(\bar w^I))) 
\end{equation}
is a rational (resp. birational) parametrization of $\V(\To)_I$, it holds that  
\begin{equation}\label{eq-B}
\Qa(\bar w):=\pi_{I}^{-1}(\tilde{\Qa}(\bar w^I))=(\bar r(\bar{w}), \bar \delta (\bar r(\bar{w}))) 
\end{equation}
is a rational (resp. birational) parametrization of $\V(\To)$. 
Now, since $\bar H$ does not depend on $\{z_i\}_{i\in I^*}$, Theorem~\ref{thm:RationalTransformation} can be adapted as follows.

\begin{theorem}\label{thm:RationalTransformationSCI}
	Let $\sys \subset \FF_m(x)$ be as in~\eqref{eq-system}, $\G$ be as in~\eqref{eq-system-new}, $\bar H(\bar z) \in \FF_m^\rho$ be the tuple of non-rational coefficients of the polynomials $F_j(x,\bar z)\in \sys$ w.r.t. $x$ depending on the variables $z_i$ with $i \in I \subsetneq \{1,\ldots,n\}$, and let 
	$$\bar r(\bar w)=(\bar r^I(\bar w^I))^e(\bar w^{I^*})\in \C(\bar w)^{n}$$
	be a rational change of variables. The following are equivalent.
	\begin{enumerate}
		\item \eqref{eq-system-new} is a system of AODEs.
		\item All the components of $\bar H(\bar r^I(\bar w^I))$ are rational.
	\end{enumerate}
\end{theorem}

Moreover, Theorem~\ref{thm-solutions} (1) can be adapted directly. 
Let us analyze the applicability of Theorem~\ref{thm-solutions} (2). 
For this purpose, we assume that $\tilde{\Qa}$ is invertible (see~\eqref{eq-A}). 
This is the case if and only if $\Qa$ is invertible (see~\eqref{eq-B}). 
Moreover, for $\bar{u}=(u_1,\ldots,u_n),$ $\bar v=(v_1,\ldots,v_m)$ it holds that 
\begin{equation}\label{eq-B-1}
\Qa^{-1}(\bar u, \bar v) = \tilde{\Qa}^{-1}(\bar u^I,\bar v)^e(\bar u^{I^*}).
\end{equation}
In this situation, Theorem~\ref{thm-solutions} (2) can be restated as follows.

\begin{theorem}\label{thm-solutions2}
	Let $\sys \subset \FF_m(x)$ be as in~\eqref{eq-system} involving radicals in the variables $z_i$ with $i \in I \subsetneq \{1,\ldots,n\}$, $\G$ be as in~\eqref{eq-system-new} and let $\tilde{\Qa}(\bar w^I)$ be a rational parametrization of $\V(\To)_I$ (see~\eqref{eq-A}) with the extension $\Qa(\bar w)=(\bar r(\bar w),\bar \delta(\bar r(\bar w)))$ (see~\eqref{eq-B}).
	It holds that
	\begin{enumerate}
		\item For every solution $\bar w(x)$ of $\G$, $$(y_1(x),\ldots,y_\ell(x)):=(r_{1,0}(\bar w(x))\ldots,r_{\ell,0}(\bar w(x)))$$ is a solution of the differential system $\sys$. 
		\item If $\Qa$ is invertible, $(y_1(x),\ldots,y_\ell(x))$ is a solution of $\sys$, $$\bar w(x):=\Qa^{-1}(y_1(x),\ldots,y_\ell^{(n_\ell)}(x),\bar \delta(y_1(x),\ldots,y_\ell^{(n_\ell)}(x)))$$
		and $\Qa(\bar w(x))$ are well--defined, then $\bar w(x)$ is a solution of $\G$.
	\end{enumerate}
\end{theorem}

\vspace*{1mm}

\noindent \textsf{Special case II.}
Now, let us deal with the case where the radicals in the differential equations in $\sys$ can be considered separately. 
More precisely, let $\sys$ be as in~\eqref{eq-system}, and let $\bar H(\bar z) \in \FF_m^\rho$ be the tuple of non-rational (function) coefficients of $F(x,\bar z)$ w.r.t. $x$. 
We assume that the entries of $\bar H$ can be grouped separately in terms of the $z_i$. 
That is, there exist $I_1,\ldots,I_s\subset \{1,\ldots,n\}$, $I_i\cap I_j=\emptyset$ for $i\neq j$, and such that the tuple $\bar H=(h_i)_{i\in \{1,\ldots,\rho\}}$ can be re-ordered as 
\[ \bar{H}=( \bar{H}^{1},\ldots,\bar{H}^{s}) \]
where $\bar{H}^{j}$ collects the entries in $\bar H$ depending on $\{ z_d \,|\, d\in I_i\}$.
Let $I_{i}^{*}=\bigcup_{j>i} I_j$. 
Furthermore we use the notation $\bar p^K$ introduced in the special case I. 
In this situation, we will proceed recursively.

Firstly, let us study the initial step.
We consider the radical tower
\[ \To^{1}=\{\mathbb{L}_1:=\C(\bar z^{I_{1}^{*}},\bar H^2, \ldots, \bar H^s), \FF_{0} = \mathbb{L}_1(\bar z^{I_1}) \subseteq \FF_1 \subseteq \cdots \subseteq \FF_{m_1}
\} \]
where $\FF_i=\FF_{i-1}(\delta_{1i})$ with $\delta_{1i}^{e_{1i}} = \alpha_{1i} \in \FF_{i-1}\setminus \mathbb{L}_1, e_{1i} \in \N$. 
Let $\bar \delta^{1}=(\delta_{11},\ldots,\delta_{1m_1})$ be the tuple collecting the $\delta_{1i}$'s involved in $\To^{1}$. 
Then, the polynomials $F$ defining~\eqref{eq-system} can be seen as polynomials in $\FF_{m_1}(x)$. 
In addition, we consider the radical parametrization
\[ \Pa^{1}:=(\bar z^{I_1}, H^{1}(\bar z^{I_1})) \]
as well as the radical variety $\V(\Pa^1)$ and its associated tower variety
\[ \V(\To^1) \subset \C^{\#(I_1)+m_1}.\]
By Lemma~\ref{lem-unirational}, we have that $\dim(\V(\To^1))=\dim(\V(\Pa^1))=\#(I^1)$.
Now, let us assume that $\V(\To^1)$ is rational, and let $\bar r^{I_1}(\bar w^{I_1})$ be the corresponding rational change of variables. 
Then we consider the associated system~\eqref{eq-system-new} where $r_{k,d}$ for $j \in I_1^{*}, z_j=y_k^{(d)}$ are  the dependent variables not treated yet in the recursive process. 

For $1<k \le s$ and after $k-1$ many steps, we obtain the radical tower
\[ \To^{k}=\{\mathbb{L}_k:=\C(\bar z^{I_k^*},\bar H^{k+1}, \ldots, \bar H^s), \FF_0 = \mathbb{L}_k(\bar z^{I_k}) \subseteq \FF_1 \subseteq \cdots \subseteq \FF_{m_k} \} \]
where $\FF_i=\FF_{i-1}(\delta_{ki})$ with $\delta_{ki}^{e_{ki}} = \alpha_{ki} \in \FF_{i-1}\setminus \mathbb{L}_k, e_{ki} \in \N$ with the radical parametrization
\[ \Pa^k:=(\bar z^{I_k}, H^k(\bar y^{I_k})) . \]
Again, the radical variety $\V(\Pa^k)$ and its associated tower variety
\[ \V(\To^k) \subset \C^{\#(I_k)+m_k} \]
have   dimension $\#(I^k)$. 
For a rational change of variables $\bar r^{I_k}(\bar w^{I_k})$, the equations in system~\eqref{eq-system-new} have to be updated.

Eventually, if all tower varieties $\V(\To^1),\ldots,\V(\To^s)$ are rational, Theorem~\ref{thm:RationalTransformation} and Theorem~\ref{thm-solutions} hold by recursively applying the components in the corresponding rational change of variables
\[ \bar r(\bar w)=(\bar r^{I_1}(\bar w^{I_1}),\ldots,\bar r^{I_s}(\bar w^{I_s}) . \]
Note that by the structure of the radicals in case II, the change of variables can be performed also in parallel.

\begin{remark}
	Note that special case I can be seen as a particular instance of special case II:
	In the notation from above, let us define $I_1$ as the index set of variables $z_i$ which occur in radical expressions in $\sys$ and set $I_2=I^*$. 
	Then the radical tower $\To^2$ is trivial and the change of variables is simply the identity, namely $\bar r^{I_2}(\bar w^{I_2})=\bar w^{I_2}$.
\end{remark}

\section{Algorithmic Treatment}\label{sec-alg}
The previous ideas are in principle algorithmic, except that for arbitrary dimension it cannot be decided whether the tower variety $\V(\To)$ is rational or not. 
Therefore, we present the following procedure in a non-algorithmic way. 
In some special situations, however, the procedure is indeed an algorithm.

\vspace*{2mm}

Let $\sys$ be a system of ordinary differential equations as in~\eqref{eq-system}. 
\begin{enumerate}
	\item Collect in the tuple $\bar H(z_1,\ldots,z_n)$ all radical expressions in~\eqref{eq-system}.
	\item If it exists, compute a partition $I_1,\ldots,I_s \subseteq \{1,\ldots,n\}$ and $\bar H = (\bar H^1,\ldots,\bar H^s)$ such that the radicals are separated (cf. Special case II in Section~\ref{sec-special}).
	\item For every tower variety $\V(\To^k)$, compute a rational parametrization $\Qa_k(\bar r^{I_k},\bar \delta^{I_k}(\bar r^{I_k}))$ if possible.
	\item Perform the change of variables $\bar z = \bar r(\bar w)$ in order to obtain a system of AODEs $\G$ as in~\eqref{eq-system-new}.
\end{enumerate}

\begin{remark}
	In the case that all the dimensions of the tower varieties $\V(\To^k)$, are either one or two, namely when the varieties $\V(\To^k)$ define a curve or surface, it can be decided whether a rational parametrization $\Qa$ exists. 
	In the affirmative case, compute such a proper (i.e. invertible) rational parametrization (see~\cite{SWP08,S1998Rational}). 
	In this way, the above procedure is indeed an algorithm for this special situation.
\end{remark}

\begin{remark}
	Usually, one is interested in working, if possible, with real solutions. 
	With the method described in this paper, if the system $\G$ has real solutions, and the tower variety can be parametrized over the field of the real numbers, then the deduced solutions of the system $\sys$ are also real. 
	If the tower variety is a curve, there exist algorithms to either directly compute a real rational parametrization (see~\cite{sendra1999algorithms}) or to reparametrize complex parametrizations into real parametrizations (see~\cite{recio1997real}). 
	If the tower variety is a surface the situation is not so complete as in the case of curves, but still there exist algorithms to compute real parametrizations for certain types of surfaces (see e.g.~\cite{andradas2011proper,andradas2009simplification,Schicho1998RationalPO,Schicho2000ProperPO}).
\end{remark}

 We finish this section illustrating the previous ideas by some examples. 
 
\begin{example}\label{ex-maple}
We consider the ordinary differential equation
\[ F:=\sqrt{y \! \left(x \right)^{2}+\left(\frac{d}{d x}y \! \left(x \right)\right)^{2}}+y \! \left(x \right)^{5}=0.
\]
First, we try to solve the equation using the command \textsf{dsolve} of the computer algebra system Maple 2023.  Maple  provides huge implicit expressions for the solutions, involving integrals, that look hard to be used; we omit here the output.

 Let us apply the method developed in the paper. The radical in the equation in  
$\sqrt{z_{1}^{2}+z_{2}^{2}}$ and the corresponding radical parametrization  is
\[ \Pa(z_1,z_2)=\left(z_1,z_2,\sqrt{z_{1}^{2}+z_{2}^{2}}\right). \]
The tower variety is  
\[  \mathbb{V}(\mathbb{T})=
\left\{ (u_1,u_2,\Delta)\in \C^3\,|\, \Delta^2 - u_{1}^2 - u_{2}^2=0 \right\} \]
which is  a rational surface that can be parametrized as 
\[ \mathcal{Q}(w_1,w_2)=\left(\frac{w_{2} \left(w_{1}^{2}-1\right)}{w_{1}^{2}+1}, 
\frac{2 w_{2} w_{1}}{w_{1}^{2}+1}, w_{2}\right). 
 \]
 So, the differential system~\eqref{eq-system-new}, associated to the equation $F(y,y')=0$ via $\mathcal{Q}$, is
\[
\left. \begin{array}{r} 
 w_{2}\! \left(x \right)^{3} \left(w_{1}\! \left(x \right)^{10} w_{2}\! \left(x \right)^{2}+w_{1}\! \left(x \right)^{10}-5 w_{1}\! \left(x \right)^{8} w_{2}\! \left(x \right)^{2}+5 w_{1}\! \left(x \right)^{8}\right.\,\,\,\,\,
  \,\,\,\,\,\\ \left. +10 w_{1}\! \left(x \right)^{6} w_{2}\! \left(x \right)^{2}+10 w_{1}\! \left(x \right)^{6}-10 w_{1}\! \left(x \right)^{4} w_{2}\! \left(x \right)^{2}+10 w_{1}\! \left(x \right)^{4} \right.\,\,\,\,\,
  \,\,\,\,\,\\ \left. +5 w_{2}\! \left(x \right)^{2} w_{1}\! \left(x \right)^{2}+5 w_{1}\! \left(x \right)^{2}-w_{2}\! \left(x \right)^{2}+1\right)=0
\\ \noalign{\vspace*{1mm}}
-w_{1}\! \left(x \right)^{4} \left(\frac{d}{d x}w_{2}\! \left(x \right)\right)+2 w_{1}\! \left(x \right)^{3} w_{2}\! \left(x \right)-4 w_{2}\! \left(x \right) w_{1}\! \left(x \right) \left(\frac{d}{d x}w_{1}\! \left(x \right)\right) \,\,\,\,\,
  \,\,\,\,\,\\  \noalign{\vspace*{1mm}}  +2 w_{2}\! \left(x \right) w_{1}\! \left(x \right)+\frac{d}{d x}w_{2}\! \left(x \right)
=0
 \end{array}\right\}\]
Now, applying the Maple command \textsf{dsolve} to the system we get the solutions
\[ (\pm i,0), (\pm 1,0), (w_1(x),0), (0,\pm 1), \left(-\tan \! \left(\dfrac{x}{3}+\dfrac{c_{1}}{3}\right), \pm b(x)\right)\]
where
\[b(x)= \frac{\left(\cos^{2}\left(\frac{x}{3}+\frac{c_{1}}{3}\right)\right) \sqrt{-\left(\sec^{2}\left(\frac{x}{3}+\frac{c_{1}}{3}\right)\right) \left(\sec^{2}\left(\frac{x}{3}+\frac{c_{1}}{3}\right)-2\right)}}{{\left(2 \left(\cos^{2}\left(\frac{x}{3}+\frac{c_{1}}{3}\right)\right)-1\right)}^{3}}. \]
Let 
\[ r(w_1,w_2)=\dfrac{w_{2} \left(w_{1}^{2}-1\right)}{w_{1}^{2}+1}.\]
Then, for each solution above, if the substitution in $r$ is well--defined, we get a solution of the initial equation. More precisely
\[
\begin{array}{ccc}
(w_1(x), w_2(x)) & & y(x) \\
\noalign{\vspace*{1mm}}
\hline
\noalign{\vspace*{1mm}}
(\pm \imath,0) & \longrightarrow & \text{no solution} \\
(\pm 1,0) & \longrightarrow &   0 \\
(w_1(x),0) & \longrightarrow & 0 \\
(0,\pm 1) & \longrightarrow &  0 \\ 
\noalign{\vspace*{1mm}}
 \left(-\tan \! \left(x +c_{1} \right), b(x) 
\right) & \longrightarrow & -\dfrac{1}{{\left(2 \left(\cos^{2}\left(\frac{x}{3}+\frac{c_{1}}{3}\right)\right)-1\right)}^{\frac{3}{2}}}\end{array}
\]
The solution 
\[  \left(-\tan \! \left(x +c_{1} \right), - b(x)
\right) \]
of~\eqref{eq-system-new} provides the solution
\[ y(x)=\frac{1}{{\left(2 \left(\cos^{2}\left(\frac{x}{3}+\frac{c_{1}}{3}\right)\right)-1\right)}^{\frac{3}{2}}} \]
of the conjugated equation 
\[ -\sqrt{y \! \left(x \right)^{2}+\left(\frac{d}{d x}y \! \left(x \right)\right)^{2}}+y \! \left(x \right)^{5}=0.\]
\end{example}

\begin{example}
	Let us consider the system of ordinary differential equations
	\begin{equation*}
	\sys = \left\{
	\begin{aligned}
	F_1 &= \sqrt{y_1^2+y_1'^2}-2x+y_2^{2/3} = 0,\\
	F_2 &= \sqrt[3]{y_2^2-y_2'}-3y_1y_1' = 0
	\end{aligned}
	\right.
	\end{equation*}
	The radical expressions are $\bar H(z_1,z_2,z_3,z_4)=(\sqrt{z_1^2+z_2^2},\sqrt[3]{z_3},\sqrt[3]{z_3-z_4})$. 
	We set $I_1=\{1,2\}$ and $I_2=\{3,4\}$. 
	The tower variety $\V(\To^1)$ is the surface given by $(z_1,z_2,\delta_1)$ where $\delta_1^2=z_1^2+z_2^2$. 
	$\V(\To^1)$ can be parametrized, for example, as $$\Qa_1(w_1,w_2)=\left(\frac{(w_1^2-1)w_2}{w_1^2+1},\frac{2w_1w_2}{w_1^2+1},w_2\right).$$
	The tower variety $\V(\To^2)$ is given by $(z_3,z_4,\delta_2,\delta_3)$ where $\delta_2^3=z_3$, $\delta_3^2=z_3^2-z_4$ and has the rational parametrization $\Qa_2(w_3,w_4)=(w_3^3, w_3^6-w_4^3,w_3,w_4)$.
	Then the system~\eqref{eq-system} is
	\begin{equation*}
	\G = \left\{
	\begin{aligned}
	&G_1=w_3^2+w_2-2x=0, \ G_2=w_1^4w_4-6w_1^3w_2^2+2w_1^2w_4+6w_2^2w_1+w_4=0, \\
	&w_2^4w_2'+(4w_1w_2-1)w_1'=2(w_1^2+1)w_1w_2, \ 3w_3^2w_3'=w_3^6-w_4^3
	\end{aligned}
	\right.
	\end{equation*}
	Let us remark that the procedure suggested in Remark~\ref{rem:naive} does not terminate in reasonable time for this example.
\end{example}

\begin{example}\label{example-fermi}
	Let us consider the Thomas-Fermi equation
	$$\sqrt{x}\,y''-y^{3/2} = 0.$$
	By the methods from~\cite{caravantes2021transforming}, the given differential equation gets transformed by $x=z^2$ into the differential equation
	$$F=z\,Y''-Y'-4z^2\,Y^{3/2} = 0.$$
	The only radical expression in $F$ is $Y^{1/2}$ (special case I), leading to a tower variety $\V(\To)$  implicitly defined by the polynomial $\Delta^2-Y$ and 	
	 parametrized by $(Y,\Delta):=(W^2,W)$. 
	By using the transformation $Y=W^2$, and directly set $Y'=2WW', Y''=2(WW''+W'^2)$, the above equation is equivalent to
	$$G=z\,W\,W''+z\,W'^2-W\,W'-2z^2\,W^3 = 0.$$
	
	By using the Newton polygon method~\cite{Cano2005}, we obtain for the given initial value $y(0)=W(0)=0$ the solution family
	$$W(z)=cz + \frac{c^2z^5}{12} + \frac{c^3z^9}{360} + \mathcal{O}(z^{13}),$$
	where $c>0$ is an arbitrary constant, and hence,
	$$y(x)=\sqrt{W(\sqrt{x})} = \sqrt{c}\sqrt{x} + \frac{c^{3/2}x^{9/2}}{24} + \frac{c^{5/2}x^{17/2}}{1920} + \mathcal{O}(x^{25/2}).$$
	For given $c>0$, the terms of order $25/2$ are uniquely given and no field extension of the coefficients is necessary anymore. 
	Hence, existence and uniqueness of the solution is guaranteed and $y(x)$ is a convergent Puiseux series in $\Q(\sqrt{c})((x^{1/2}))$.
	
	Let us note that for the frequently used initial values $y(0)=1, y(\infty)=0$ there is no formal Puiseux series solution and other solution methods as the Newton polygon method have to be used for finding solutions of $G=0$ and hence, of the Thomas-Fermi equation itself. 
\end{example}

\section{Partial Differential Equations}\label{sec-PDE}
In this section we treat systems of partial differential equations $\sys$ in the differential indeterminates $x_1,\ldots,x_q$. 
More precisely, let $\sys$ consist of differential equations of order at most $s$ and possibly involving radicals in $\Theta(\alpha)y_k:=\frac{\partial^{|\alpha|} y_k}{\partial x_1^{\alpha_1} \cdots \partial x_q^{\alpha_q}}$ with $1 \le k \le \ell$, $|\alpha| \le s$. 
Let $z_1,\ldots,z_n$ enumerate the $\Theta(\alpha)y_k$ and consider a radical tower~\eqref{eq-tower}. 
We assume that $\sys \subset \FF_m(x_1,\ldots,x_q)$.

We are looking for a change of variables~\eqref{eq-substitution} such that $F(\bar x,\bar r(\bar w))$ is an algebraic equation for every $F$ in $\sys$. 
Let $z_i=\Theta(\alpha)y_k, z_j=\Theta(\alpha+e_p)y_k$ for some $|\alpha|<s$ and a unit vector $e_p$ which is one exactly at the $p$-th position. 
The differential relations are kept by additionally imposing all such
\begin{equation}\label{eq-chainRulePartial}
\sum_{t=1}^n \frac{\partial\,r_i}{\partial w_t}(\bar w) \cdot \frac{\partial\,w_t}{\partial x_p} = r_j(\bar w).
\end{equation}

In this way, Theorem~\ref{thm:RationalTransformation} and Theorem~\ref{thm-solutions} hold in the expected way as for the case of ODEs (where $q=1$). 
Let us illustrate this in the following example.

\begin{example}\label{ex-partial}
	Let us consider the system of partial differential equations
	\begin{equation*}
	\sys = \left\{
	\begin{aligned}
	F_1 &= \sqrt[3]{\sqrt{z_1-1}+z_3}+x_1\,(z_1+\sqrt{z_5}) = 0,\\
	F_2 &= \sqrt{z_2+1}^3+x_2^2\,z_1z_4 = 0,\\
	F_3 &= x_1\,\sqrt{z_1z_4-z_2z_3}+x_2\,z_5^2-1 = 0
	\end{aligned}
	\right.
	\end{equation*}
	where $(z_1,z_2,z_3,z_4,z_5)=(y_1,\frac{\partial\,y_1}{\partial x_1},\frac{\partial\,y_1}{\partial x_2},\frac{\partial^2\,y_1}{\partial x_1 \partial x_2},y_2)$. 
	The radical coefficients are 
	collected in 
	$$\bar{H}=\left(\sqrt[3]{\sqrt{z_1-1}+z_3},\sqrt{z_5},\sqrt{z_2+1}^3,\sqrt{z_1z_4-z_2z_3}\right).$$
	We set $I_1=\{1,2,3,4\}$ and $I_2=\{5\}$. 
	The tower variety $\V(\To^1)$ is given by $(\bar z^{I_1},\delta_1,\ldots,\delta_4)$ where $\delta_1^2=z_1-1, \delta_2^3=\delta_1+z_3, \delta_3^2=z_1z_4-z_2z_3, \delta_4^2=z_2+1$. 
	A parametrization of $\V(\To^1)$ is
	\begin{align*}
	\Qa_1(\bar w^{I_1}) = \big(&w_1^2+1,w_4^2-1, w_2^3-w_1,\frac{w_2^3w_4^2-w_1w_4^3-w_2^3+w_3^2+w_1}{w_1^2+1},\\
	&w_1,w_2,w_3,w_4 \big).
	\end{align*}
	Additionally considering the transformation $\Qa_2(w_5)=(w_5^2,w_5)$, and the differential relations, we obtain the system
	\begin{equation*}
	\G = \left\{
	\begin{aligned}
	G_1 &= w_2+x_1\,(w_1^2+1+w_5) = 0,\\
	G_2 &= w_4^3+x_2^2(w_2^3w_4^2-w_1w_4^3-w_2^3+w_3^2+w_1) = 0,\\
	G_3 &= x_1w_3+x_2w_5^4-1 = 0,\\
	&2w_1 \cdot \frac{\partial\,w_1}{\partial x_1} = w_4^2-1, \quad
	2w_1 \cdot \frac{\partial\,w_1}{\partial x_2} = w_2^3-w_1, \\
	&\left(3w_2^2 \cdot \frac{\partial\,w_2}{\partial x_1}-\frac{\partial\,w_1}{\partial x_1}\right) \cdot (w_1^2+1) = w_2^3w_4^3-w_1w_4^3-w_2^3+w_3^2+w_1, \\
	&2w_4(w_1^2+1) \cdot \frac{\partial\,w_4}{\partial x_2} = w_2^3w_4^3-w_1w_4^3-w_2^3+w_3^2+w_1
	\end{aligned}
	\right.
	\end{equation*}
\end{example}

\section{Declarations}
 \begin{itemize}
\item Ethical Approval: not applicable
\item Competing interests: The authors declare that they have no competing interests.
\item Authors' contributions: Both authors contributed equally and significantly in writing this paper.
\item Availability of data and materials: not applicable
 \end{itemize}
 
\appendix

\bibliographystyle{acm}

\end{document}